\documentclass[12pt]{article}
\usepackage[utf8]{inputenc}

\usepackage{amsmath,amsfonts,amssymb,latexsym,amsthm,verbatim,a4wide,url,appendix}

 \theoremstyle{plain}

\newtheorem{theorem}{Theorem}[section]
\newtheorem{lemma}[theorem]{Lemma}
\newtheorem{proposition}[theorem]{Proposition}
\newtheorem{definition}[theorem]{Definition}
\newtheorem{remark}[theorem]{Remark}
\newtheorem{example}[theorem]{Example}
\newtheorem{corollary}[theorem]{Corollary}
\newtheorem{condition}{Condition}

\newcommand{\ep}{\mathbb{E}}
\newcommand{\pr}{\mathbb{P}}
\newcommand{\re}{\mathbb{R}}
\newcommand{\MM}{{\mathcal M}}
\newcommand{\ol}[1]{\overline{#1}}

\title{Information-theoretic convergence of extreme values \\ to the Gumbel distribution}
\author{Oliver Johnson\thanks{School of Mathematics, University of Bristol, Fry Building,
Woodland Road, Bristol, BS8 1UG, UK
E-mail: {\tt O.Johnson@bristol.ac.uk}}}
\date{\today}

\begin{document}

\maketitle

\begin{abstract} \noindent
We show how convergence to the Gumbel distribution in an extreme value setting can be understood in an information-theoretic sense. We  introduce a new type of score function which behaves well under the maximum operation, and which implies simple expressions for entropy and relative entropy. We show that, assuming certain properties of the von Mises representation, convergence to the Gumbel can be proved in the strong sense of relative entropy.
\end{abstract}


\section{Introduction and notation}

It is well-known that convergence to the Gaussian distribution in the Central Limit Theorem regime can be understood in an information-theoretic sense, following the work of Stam \cite{stam}, Blachman \cite{blachman}, Brown \cite{brown}, and in particular Barron \cite{barron} who proved convergence in relative entropy (see \cite{johnson14} for an overview of this work). While a traditional characteristic function proof of the Central Limit Theorem may not give a particular insight into why the Gaussian is the limit, this information-theoretic argument (which can be understood to relate to Stein's method \cite{stein2}) offers an insight on this.

To be specific, we can understand this convergence through the (Fisher) score function with respect to location parameter 
$\rho_X(x) = f_X'(x)/f_X(x) = \left( \log f_X(x) \right)'$ of a random variable $X$ with density $f_X$, where $'$ represents the spatial derivative. Two key observations are (i) that a standard Gaussian random variable $Z$ is characterized by having linear score $\rho_Z(x) = - x$ and (ii) there is a closed form expression for the score of the sum of independent random variables as a conditional expectation (projection) of the scores of the individual summands (see e.g. \cite{blachman}). As a result of this, the score function becomes `more linear' in the Central Limit Theorem regime (see \cite{johnson5, johnson50}). Similar arguments can be used to understand `law of small numbers' convergence to the Poisson distribution \cite{johnson11}.

However, there exist other kinds of probabilistic limit theorems which we would like to understand in a similar framework. In this paper we will consider a standard extreme value theory setup \cite{resnick2}: that is, we take i.i.d. random variables $X_1, X_2, \ldots \sim X$ and define $M_n = \max( X_1, \ldots, X_n)$ and $N_n = (M_n - b_n)/a_n$ for some normalizing sequences  $a_n$ and $b_n$. We want to consider whether $N_n$ converges in relative entropy to a standard extreme value distribution. This type of extreme value analysis naturally arises in a variety of contexts including the modelling of natural hazards, world record sporting performances and applications in finance and insurance.

In this paper we show how to prove convergence in relative entropy for the case of a Gumbel (Type I Extreme Value) limit, by introducing a different type of score function, which we refer to as the max-score $\Theta_X$, and which is designed for this problem. Corresponding properties to those described above hold for this
new quantity: (i) a Gumbel random variable $X$ can be characterized by having linear max-score $\Theta_X$  (see Example \ref{ex:gumbel}) (ii) there is a closed form expression (Lemma \ref{lem:scoreval}) for the max-score of the maximum of independent random variables.

In Section \ref{sec:scoreent} we show that the entropy and relative entropy can be expressed in terms of the max-score, 
in Section \ref{sec:exmaxscore} we show how to calculate the expected value of the max-score in the maximum regime,
and in Section \ref{sec:vonmises} we relate this to the standard von Mises representation (see \cite[Chapter 1]{resnick2}) to deduce convergence in relative entropy in Theorem \ref{thm:main}. Our aim is not to provide a larger class of random variables than papers such as \cite{ dehaan2, dehaan, pickands} for which convergence to the Gumbel takes place, but rather to use ideas from information theory to understand why this convergence may be seen as natural, and to prove convergence in a strong (relative entropy) sense. So while for example it is known that the standardized maximum converges in total variation (see for example \cite[p.159]{reiss}), by Pinsker's inequality we know that convergence in relative entropy is stronger.

We briefly remark that entropy was studied in the Gumbel convergence regime by Saeb \cite{saeb}, though using a direct computation based on the density. The standard Fisher score was used in an extreme value context by Bartholm\'{e} and Swan \cite{bartholme} in a version of Stein's method. Extreme value distributions were considered in the context of Tsallis entropy by Bercher and Vignat \cite{bercher}. However, this particular max-score framework is new, to the best of our knowledge.

\begin{definition}
For absolutely continuous random variable $Z \in \re$ we write distribution function $F_Z(z) = \pr(Z \leq z)$, tail distribution function $\ol{F}_Z(z) = \pr(Z > z) = 1-  F_Z(z)$, density $f_Z(z)=F_Z'(z)$ and harazrd function $h_Z(z) = f_Z(z)/ol{F}_Z(z)$. We  define the max-score function:
\begin{eqnarray} \Theta_Z(z) & := & \log( f_Z(z)/F_Z(z)),\label{eq:scoredef} \\
& = & \log h_Z(z) + \log(1-F_Z(z)) - \log F_Z(z), \label{eq:vonmisesscore}
\end{eqnarray}
where the second result follows on rearranging. 
\end{definition}

Note that we can write $F_X(x) = \exp(-w(x))$ for some decreasing function $w$ with $w(\infty) = 0$. Then $f_X(x)  
 = - w'(x) F_X(x)$ so that $\Theta_X(x) = \log( - w'(x))$.
Equivalently, inverting this relationship gives  \begin{equation} F_X(x) = \exp \left(  - \int_x^\infty e^{\Theta_X(u)} du \right).
\label{eq:invert}
\end{equation}
We now remark that under this definition the Gumbel distribution has linear max-score function:

\begin{example} \label{ex:gumbel}
A Gumbel random variable $Y$ with parameters $\mu$ and $\beta$ has distribution function $F_Y(y) = \exp \left( - e^{-(y- \mu)/\beta} \right)$, so in the notation above
$w(y) = e^{-(y-\mu)/\beta}$. Hence $w'(y) = -e^{-(y-\mu)/\beta}/\beta$ and a Gumbel random variable has max-score 
$$\Theta_Y(y) = - \log \beta - \frac{(y - \mu)}{\beta}. $$
Indeed, using \eqref{eq:invert}, we can see the property of having a linear max-score $\Theta_Y$ characterizes the Gumbel.

For future reference in this paper, note that (see \cite[Eq. (1.25)]{kotz}) the  $\ep Y = \mu+ \beta \gamma$, where $\gamma$ is the Euler--Mascheroni constant, and $\MM_Y(t) =  e^{\mu t} \Gamma(1-\beta t )$ (see \cite[Eq. (1.23)]{kotz}).
\end{example}

We can further state how the max-score function behaves under the maximum and rescaling operations:

\begin{lemma} \label{lem:scoreval} If we write $M_n = \max( X_1, \ldots, X_n)$ and $N_n = (M_n - b_n)/a_n$ then
\begin{eqnarray} 
\Theta_{N_n}(z)  & = & \log(n a_n) + \Theta_X(a_n z + b_n),  \label{eq:scoreval} \\
\Theta_{N_n}(N_n)  & = & \log(n a_n) + \Theta_X(M_n),  \label{eq:scoreval2}
\end{eqnarray}
\end{lemma}
\begin{proof}
As usual (see for example \cite[Chapter 0.3]{resnick2}), we know that by independence
\begin{equation} \label{eq:mnpower}
F_{M_n}(x) = \pr \left( \max( X_1, \ldots, X_n) \leq x \right) = \pr \left( \bigcap_{i=1}^n \{ X_i \leq x \} \right) = F_X(x)^n,
\end{equation}
 so that $F_{N_n}(x) = F_{M_n}(a_n x + b_n) = F_{X}(a_n x + b_n)^n$.
This means that $f_{N_n}(x) = n a_n F_X(a_n x + b_n)^{n-1} f_X(a_n x + b_n)$,
so $f_{N_n}(x)/F_{N_n}(x) = n a_n f_X(a_n x + b_n)/F_X(a_n x + b_n)$
 and \eqref{eq:scoreval} follows on taking logarithms. The second result, \eqref{eq:scoreval2}, follows by direct substitution using the fact that $M_n = a_n N_n + b_n$.
\end{proof}

\begin{example}
In particular, if $X$ is exponential with parameter $\lambda$, so with $a_n = 1/\lambda$ and $b_n = \log n/\lambda$
$$ \Theta_X(z) = \log \left( \frac{ \lambda e^{-\lambda z} }{1 - e^{-\lambda z}} \right),$$
this gives
$$ \Theta_{N_n}(z) = \log (n/\lambda) + \log \left( \frac{ \lambda e^{-z}/n}{1 - e^{-z}/n} \right)
 =  - z - \log(1-e^{-z}/n) .$$
Hence, letting $n \rightarrow \infty$, we know that $\Theta_{N_n}(z)$ converges pointwise to $-z$, which is the max-score of the standard Gumbel (with parameters $\mu = 0$ and $\beta =  1$) -- see Example \ref{ex:gumbel} above. \end{example}

However, while this gives us some intuition as to why the Gumbel is the limit in this case, pointwise convergence of the score function does not seem a particularly strong sense of convergence. We now discuss the question of convergence in relative entropy.

\section{Max-score function and entropy} \label{sec:scoreent}

We next show that we can use the max-score function to give an alternative formulation for the entropy of a random variable, which allows us to quickly find the entropy of a Gumbel distribution. We first state a simple lemma, which follows directly from the fact that both $F_X(X)$ and $1-F_X(X)$ are uniformly distributed:
\begin{lemma} \label{lem:changevar}
 For any continuous random variable $X$ with distribution function $F_X$:
$$ \ep \log F_X(X)  = \ep \log \left( 1- F_X(X) \right)  = -1.$$
\end{lemma}
%

\begin{proposition} \label{prop:entval}
For an absolutely continuous random variable $X$ with max-score function $\Theta_X$, the entropy $H(X)$ satisfies
\begin{equation} H(X) = 1 - \ep \Theta_X(X). \label{eq:keyrep} \end{equation}
\end{proposition}
\begin{proof} The key observation is that
$\log f_X(x) = \log F_X(x) + \Theta_X(x)$ so that:
\begin{eqnarray}
H(X)  & =  & - \int_{-\infty}^{\infty} f_X(x) \log f_X(x) dx  \nonumber  \\
& =  & -\int_{-\infty}^{\infty} f_X(x) \log F_X(x) dx - \int_{-\infty}^{\infty} f_X(x) \Theta_X(x) dx \label{eq:keyrepa} \\
& = &  1 - \ep \Theta_X(X), \nonumber 
\end{eqnarray}
where we apply Lemma \ref{lem:changevar} to find the first term of \eqref{eq:keyrepa}.
\end{proof}

In particular we recover the entropy of $Y$, a Gumbel distribution (see for example \cite[Theorem 1.6iii)]{ravi2}):
\begin{example}  For $Y$, a Gumbel distribution with parameters $\mu$ and $\beta$,  using Example \ref{ex:gumbel}
$$ H(Y) = 1 - \ep \left(  - \log \beta -  \frac{Y - \mu}{\beta} \right) =1 + \log \beta + \gamma,$$
since  $\ep Y = \mu+ \beta \gamma$.
 \end{example}

 We can use similar arguments to give an expression for the relative entropy $D(X \| Y)$ where $Y$ is Gumbel: 

\begin{proposition} Given absolutely continuous random variable $X$, we can write the relative entropy from $X$ to $Y$, a Gumbel random variable with parameters $\mu$ and $\beta$, as
\begin{equation} \label{eq:keyDrep}
D(X \| Y)  =  \left( \ep \Theta_X(X) + \log \beta  + \frac{\ep X-\mu}{\beta} \right) + \left( \ep e^{-(X-\mu)/\beta} - 1 \right),.
\end{equation}
assuming both sides of the expression are finite.
\end{proposition}
\begin{proof}
\begin{eqnarray*}
D( X \| Y ) & = & \int f_X(x) \log \left( \frac{ f_X(x)}{f_Y(x)} \right) dx \\
& = & - H(X) - \int f_X(x) \log f_Y(x) dx \\
& = & - \left( 1 - \ep \Theta_X(X) \right) - \int f_X(x) \log F_Y(x) dx - \int f_X(x) \Theta_Y(x) dx \\
& = &  \ep \Theta_X(X)  -1 + \int f_X(x) \exp \left(- \frac{(x-\mu)}{\beta} \right) dx + \log \beta + \int f_X(x) \frac{(x-\mu)}{\beta} dx.
\end{eqnarray*}
substituting from Proposition \ref{prop:entval} and using the value of $\Theta_Y$ from Example \ref{ex:gumbel}.
\end{proof}

Observe that in the case of $X$ itself Gumbel with the same parameters as $Y$, both bracketed terms in \eqref{eq:keyDrep} vanish: 
\begin{enumerate}
\item
We can rewrite the first term as $\ep \left( \Theta_X(X) - \Theta_Y(X) \right)$, using the value of the max-score in the Gumbel case (Example \ref{ex:gumbel}). This suggests that (as in \cite{johnson14}) we may wish to consider this term as a standardized score function with the relevant linear term subtracted off.
\item We can rewrite the second term in \eqref{eq:keyDrep} as $e^{\mu/\beta} \MM_X(-1/\beta) - 1 $, where $\MM_X(t)$ is the moment generating function. Since (see Example \ref{ex:gumbel}) the moment generating function of a Gumbel random variable is $\MM_Y(t) =  e^{\mu t} \Gamma(1-\beta t )$, we know that in the Gumbel case $e^{\mu/\beta} \MM_X(-1/\beta) - 1
=  e^{\mu/\beta} e^{-\mu/\beta} \Gamma(2) - 1   = 0$.
\end{enumerate}

\section{Expected max-score of the standardized maximum} \label{sec:exmaxscore}

We now consider the behaviour of the expected max-score of the standardized maximum $N_n = (M_n-b_n)/a_n$, using the representation \eqref{eq:vonmisesscore}. We first state a technical lemma which holds for all continuous random variables $X$:

\begin{lemma} \label{lem:technical}
 For $M_n$ the maximum of $n$ independent copies of absolutely continuous random variable $X$:
\begin{enumerate} 
\item   \label{lem:1overn} The expected value
$$ \ep \log F_X(M_n)  = -\frac{1}{n}$$
is the same for all $F_X$.
\item \label{lem:transfer}
The expected value
$$ \ep \log(1-F_X(M_n)) = -H_n$$
is the same for all $F_X$, where we write $H_n := 1 + 1/2 + 1/3 + \ldots + 1/n$ to be the $n$th harmonic number. 
\end{enumerate}
\end{lemma}
\begin{proof}
Part \ref{lem:1overn} is a simple corollary of Lemma \ref{lem:changevar}. Recalling from \eqref{eq:mnpower} that $F_{M_n}(x) = F_X(x)^n$ we know from Lemma \ref{lem:changevar} that
$$ - 1 = \ep \log F_{M_n}(M_n) = n \ep \log F_X(M_n),$$
and the result follows on rearranging.

Part \ref{lem:transfer} requires a slightly more involved calculation. By standard manipulations, we know that $-\log(1-F_X(X))$ is exponential with parameter 1. This follows by direct substitution, since
\begin{eqnarray*}
\pr( -\log(1-F_X(X) \leq z) & = & \pr \left( F_X(X) \leq 1 - \exp(-z) \right) \\
& = &  \pr \left( X \leq F_X^{-1}(1-\exp(-z) \right) = F_X \left( F_X^{-1}(1-\exp(-z) \right) \\
& = & 1- \exp(-z), 
\end{eqnarray*}
as required. Now, since $-\log(1-F_X(t))$ is increasing in $t$, we can write
\begin{eqnarray*}
-\log(1-F_X(M_n))  & = & -\log \left(  1 - F_X \left( \max_{1 \leq i \leq n} X_i \right) \right) 
= \max_{1 \leq i \leq n} -\log(1-F_X(X_i))  \sim \max_{1 \leq i \leq n} E_i,
\end{eqnarray*}
where $E_i$ are independent exponentials with parameter 1.

It is well-known that the expected value of $\max_{1 \leq i \leq n} E_i  =  H_n$, the $n$th harmonic number. The simplest proof of this is  to write
$$ \max_{1 \leq i \leq n} E_i = \sum_{i=1}^n U_i, $$ where $U_i$ are independent exponentials with parameter $n- i+1$.  (This follows from the memoryless property of $E_i$, by thinking of $U_1$ as the time for the first exponential event to happen, $U_2$ as the time for the second, and so on). Since $\ep U_i = 1/(n-i+1)$, the result follows.
\end{proof}

We can put this together to deduce that:
\begin{lemma} \label{lem:firstterm}
 For any absolutely continuous $X$, writing $N_n = (M_n - b_n)/a_n$ for any sequence of norming constants
$a_n$, $b_n$ we deduce that
$$ \ep \Theta_{N_n}(N_n) =  \left( \log a_n + \ep \log h_X(M_n) \right) +  \left( \log n  - H_n \right)  +  \frac{1}{n}.$$
\end{lemma}
\begin{proof} 
Using the representation  \eqref{eq:scoreval2} of the score in  Lemma \ref{lem:scoreval} and the expression \eqref{eq:vonmisesscore} we know that
\begin{eqnarray*}
\ep \Theta_{N_n}(N_n) 
& = & \log(n a_n) + \ep \Theta_X(M_n)   \nonumber \\
& = &  \log(n a_n)  + \ep \log h_X(M_n) +  \ep \log(1-F_X(M_n))- \ep \log F_X(M_n)  \nonumber \\
& = &   \left( \log a_n + \ep \log h_X(M_n) \right) +  \left( \log n  - H_n \right)  +  \frac{1}{n}, 
\end{eqnarray*}
using the two parts of Lemma \ref{lem:technical}.
\end{proof}
Note that only the first bracketed term of Lemma \ref{lem:firstterm} depends on the particular choice of $X$.

\section{von Mises representation and convergence in relative entropy} \label{sec:vonmises}
We will demonstrate convergence of relative entropy in a restricted version of the domain of maximum attraction. In order to work in terms of relative entropy, we need to assume that $X$ is absolutely continuous. Additionally, we recall the definition of a distribution function $F_X$ having a representation of von Mises type \cite[Eq.  (1.5)]{resnick2}:
\begin{definition} \label{def:vonmises}
Assume that the upper limit of the support of $X$ is $x_0 := \sup \{x: F_X(x) < 1 \}$ (which may be finite or infinite)  and
\begin{equation} \label{eq:vonmises}
 F_X(x) =  1- c(x) \exp \left( - \int_{z_0}^x \frac{1}{g(u)} du \right)
= 1- c(x) \exp \left( - G(x) \right), \end{equation}
for some auxiliary function $g$ such that $g'(x) \rightarrow 0$ as $x \rightarrow x_0$, and $\lim_{x \rightarrow x_0} c(x) = c > 0$.
\end{definition}

Assuming the von Mises representation \eqref{eq:vonmises} holds we can write the density $f_X(x) = (c(x) G'(x) - c'(x)) \exp(- G(x))$, or on dividing by $1-F_X(x) = c(x) \exp(-G(x))$ we can deduce that the hazard function satisfies
\begin{equation} \label{eq:vonmisesf}
h_X(x) = \frac{1}{g(x)} - \frac{c'(x)}{c(x)}.
\end{equation}

The canonical choice of norming constants  is given in  \cite[Proposition 1.1(a)]{resnick2} (see also \cite[Table 3.4.4]{embrechts})
as $(a_n,b_n)$ satisfying $1/n = \ol{F}(b_n)$ and $a_n = g(b_n)$.
Note that  (see \cite[Proposition 1.4]{resnick2}) the normalized maximum $N_n = (M_n - b_n)/a_n$ converges in distribution to the Gumbel 
if and only if the representation \eqref{eq:vonmises} holds.  See \cite[Table 3.4.4]{embrechts} for a  list of eight types of distributions whose standardized maximum converges to the Gumbel, some of which we give as examples below:

\begin{example} \label{ex:test} We can illustrate the representation \eqref{eq:vonmises} as follows:
\begin{enumerate}
\item For the exponential we can take $c(x)=1$, $z_0  = 0$, $x_0 = \infty$, $g(u)=1/\lambda$ and $b_n = \log n/\lambda$.
\item For the gamma distribution with shape parameter $\alpha$ and rate parameter $\beta$ we can take $c(x) = 1$, $z_0 = 0$, $x_0 = \infty$ and $g(u) = \Gamma(\alpha, \beta u)/\left(\beta (\beta u)^{\alpha-1} \exp(-\beta u) \right)$, where $\Gamma( \cdot, \cdot)$ is the upper incomplete gamma function. Note that  as $u \rightarrow \infty$ we know that $g(u) \rightarrow 1/\beta$.
\item For the standard Gaussian distribution, we can take $g(x) = (1-\Phi(x))/\phi(x)$ (for $\phi$ and $\Phi$ the standard normal density and distribution functions), and note that the Mills ratio $g(x) \simeq 1/x$ as $x \rightarrow \infty$.
\item For the `Weibull-like' distribution function of \cite[Table 3.4.4]{embrechts} with $\ol{F} \sim K x^\alpha \exp(-c x^\tau)$ (with $\tau > 0$ and $\alpha \in \re$), we can take $G(x)  = c x^\tau - \alpha \log x$, so that $g(x) = x/(\tau c x^\tau - 1)$.
\item For the Benktander-type-II example  of \cite[Table 3.4.4]{embrechts}, with $\ol{F} = x^{\beta-1} \exp(-\alpha(x^\beta-1)/\beta)$, for $\alpha > 0$, $0 < \beta < 1$, we can take  $c(x)= 1$, $z_0=1$ and $g(x) = x/(1-\beta + \alpha x^{\beta})$.
\item For the example of $F_X(x) = 1 - \exp(-x/(1-x))$ given by Gnedenko \cite{gnedenko4} (see also
 \cite[P.39]{resnick2}) we can take
$c(x)=1$, $z_0 = 0$, $x_0 = 1$, $g(u) =  (1-u)^2$ and $b_n = \log n/(1+\log n)$. This is an example of `exponential behaviour at $x_0$' in the sense of \cite[Table 3.4.4]{embrechts}, where we can take $g(x) = (x_0 - x)^2/\alpha$.
\end{enumerate}
\end{example}

We now state a restricted technical condition which we can use to give a simple proof of  convergence in relative entropy:

\begin{condition} \label{cond:main2}
\mbox{ }
\begin{enumerate}
\item \label{it:tail} Assume $\ell(t) := 1 - c'(t) g(t)/c(t) \rightarrow 1$ as $t \rightarrow x_0$.
\item \label{it:tail2} Assume there exists a constant $\sigma < 1$ such that $\log \left( g(x)/x^\sigma \right)$ is bounded and continuous, and that $\gamma := \lim_{x \rightarrow x_0} g(x)/x^\sigma$  is finite and non-zero.
\item  \label{it:moment} Assume that $\int_{-\infty}^0 |x|^k dF_X(x) < \infty$ for all $k$.
\end{enumerate}
\end{condition}

Note that Condition \ref{cond:main2}.\ref{it:tail} holds automatically with equality when  $c(x)$ is constant, which is the restricted version of the von Mises condition stated as \cite[Eq. (1.3)]{resnick2}, and which includes all but the Weibull-like part of  Example \ref{ex:test}. Note that Condition \ref{cond:main2}.\ref{it:tail2} is satisfied for the first five examples in Example \ref{ex:test} (taking $\sigma = 0$ for the exponential and gamma examples, $\sigma=-1$ for the Gaussian, $\sigma = 1-\tau$ for the Weibull-like distribution and $\sigma = 1-\beta$ for the Benktander-type-II distribution). We discuss how the analysis can be adapted in the final (`exponential behaviour at $x_0$') example in Remark \ref{rem:extra} below.

\begin{lemma} \label{lem:transfer2} Under  Condition \ref{cond:main2}.\ref{it:moment}:
\begin{enumerate}
\item The mean $\ep N_n$ converges to the Euler--Mascheroni constant $\gamma$.
\item The moment generating function converges as follows:
\begin{equation} \label{eq:mgfconv}  \lim_{n \rightarrow \infty}
\MM_{N_n}(t)
= \Gamma(1-t),\end{equation}
by Taylor's theorem.
\end{enumerate}
\end{lemma}
\begin{proof}

Note that (see \cite[Proposition 2.1]{resnick2}) under Condition \ref{cond:main2}.\ref{it:moment} the $k$th moment of $N_n$ converges:
\begin{equation} \label{eq:momconv} \lim_{n \rightarrow \infty} \ep (N_n)^k = (-1)^k \Gamma^{(k)}(1),\end{equation}
where $\Gamma^{(k)}(x)$ is the $k$th derivative of the $\Gamma$ function at $x$.

We deduce convergence of the moment generating function
\begin{eqnarray*}  \lim_{n \rightarrow \infty}
\MM_{N_n}(t) =  \lim_{n \rightarrow \infty} \sum_{r=0}^\infty \frac{t^r \ep(N_n)^r}{r!} =  \lim_{n \rightarrow \infty} \sum_{r=0}^\infty \frac{t^r \ep(N_n)^r}{r!} = \sum_{r=0}^\infty \frac{(-t)^r \Gamma^{(k)}(1)}{r!}
= \Gamma(1-t),\end{eqnarray*}
by Taylor's theorem.
\end{proof}

\begin{theorem} \label{thm:main} If the distribution function of $X$ satisfies the
von Mises representation \eqref{eq:vonmises} with $x_0 = \infty$ and Condition \ref{cond:main2} holds, there exist norming constants $a_n$ and $b_n$ satisfying $1/n = \ol{F}(b_n)$ and $a_n = g(b_n)$ such that $N_n = (M_n - b_n)/a_n$ satisfies
$$ \lim_{n \rightarrow \infty} D(  N_n \| Y) = 0,$$
where $Y$ is a standard Gumbel distribution (with $\beta = 1$ and $\mu = 0$). \end{theorem}
\begin{proof}
We use the norming constants from \cite[Proposition 1.1(a)]{resnick2} (see also \cite[Table 3.4.4]{embrechts})
We can write the first term in the relative entropy expression \eqref{eq:keyDrep} in the case $\mu = 0$ and $\beta = 1$ using Lemma \ref{lem:firstterm} as
\begin{eqnarray}
\lefteqn{
\ep \Theta_{N_n}(N_n) + \ep N_n   } \nonumber \\
& = &   \left( \log a_n + \ep \log h_X(M_n) \right) +  \left( \log n  - H_n \right)  +  \frac{1}{n} +   \ep N_n. \label{eq:firstterm}
\end{eqnarray}
We can consider the behaviour as $n \rightarrow \infty$ of the four terms in \eqref{eq:firstterm} separately:
\begin{enumerate}
\item  We can write the first term in \eqref{eq:firstterm} in terms of the $\sigma$ of Condition \ref{cond:main2}.\ref{it:tail2} as
\begin{eqnarray}  \lefteqn{  \log a_n + \ep \log h_X(M_n)  } \nonumber \\ 
& = & \log g(b_n) +  \ep \log h_X(M_n) \nonumber \\  
& = & 
\log \left( \frac{g(b_n)}{b_n^\sigma} \right) - \ep \log \left( \frac{g(M_n)}{M_n^\sigma} \right)  - \sigma \ep \log \left( \frac{M_n}{b_n} \right) 
+ \ep \log \left( g(M_n) h_X(M_n) \right).\;\; \label{eq:firstterm1a}
\end{eqnarray}
\begin{enumerate}
\item Since $b_n \rightarrow x_0$ and $M_n \rightarrow x_0$ in distribution, we know that the first two terms of \eqref{eq:firstterm1a} tend to 
$\log \gamma - \log \gamma = 0$, using the portmanteau lemma. 
\item We can control the third term of \eqref{eq:firstterm1a} by writing $M_n = b_n + a_n N_n$ (and recalling that $a_n = g(b_n)$) to obtain
\begin{eqnarray*}
\ep \log \left( \frac{M_n}{b_n} \right)  & = & \ep \log \left( 1 + \frac{a_n N_n}{b_n} \right) =
 \sum_{k=1}^\infty \frac{(-1)^k}{k}  \left( \frac{a_n}{b_n} 
\right)^k \ep N_n^k. 
\end{eqnarray*}
and we can use the facts that $a_n/b_n = g(b_n)/b_n \sim \gamma b_n^{\sigma -1} \rightarrow 0$ and (by \eqref{eq:momconv})  that $\ep N_n^k$ converges to a finite constant to deduce that this term tends to zero.
\item 
Using  the representation of the hazard function \eqref{eq:vonmisesf} we know that the fourth term of \eqref{eq:firstterm1a}
equals $ \ep \log \left( 1 - c'(M_n) g(M_n)/c(M_n) \right) = \ep \log \ell(M_n)$, so 
since $M_n \rightarrow x_0$ in distribution we know that $\ep \log \left( g(M_n) h_X(M_n) \right) \rightarrow \log \ell(x_0) = 0$, 
by the portmanteau lemma.
\end{enumerate}
Hence overall, the first term of \eqref{eq:firstterm} tends to zero.
\item It is a standard result that $\log n - H_n$ is a monotonically increasing sequence which converges to $-\gamma$.
\item Clearly the third term  in \eqref{eq:firstterm} converges to zero.
\item  Lemma \ref{lem:transfer2}  tells us that the final term converges to $\gamma$.
\end{enumerate}
Putting this all together, we deduce that \eqref{eq:firstterm} converges to
$$  0   - \gamma + 0 + \gamma = 0.$$ 

In the case $\mu = 0$ and $\beta = 1$, the second term in the relative entropy expression \eqref{eq:keyDrep} becomes
$$ \ep e^{-N_n} - 1 = \MM_{N_n}(-1) - 1 \rightarrow \Gamma(2) - 1 = 0,$$
by \eqref{eq:mgfconv}.
\end{proof}

\begin{corollary} \label{cor:ent}
Assume that the distribution function $F_X$ has a von Mises representation \eqref{eq:vonmises} whose auxiliary function $g$ satisfies Condition \ref{cond:main2}. Then the entropy of the normalized maximum $N_n = (M_n - b_n)/a_n$ satisfies
$$ \lim_{n \rightarrow \infty} H(N_n) = 1+ \gamma,$$
which is the entropy of the corresponding Gumbel distribution.
\end{corollary}
\begin{proof} By Proposition \ref{prop:entval}
\begin{eqnarray*}
 H(N_ n) & = & 1 - \ep \Theta_{N_n}(N_n) \\
& = & \left( 1 + \ep N_n \right)  - \ep \left( \Theta_{N_n}(N_n) + N_n \right)
\end{eqnarray*}
The first term converges to $1+\gamma$ by Lemma \ref{lem:transfer2}, the second term is precisely \eqref{eq:firstterm} and converges to zero as described in the proof of Theorem \ref{thm:main}.
\end{proof}

\begin{remark}
Note that for the exponential case of Example \ref{ex:test}, Condition \ref{cond:main2} is satisfied, so we can deduce convergence in relative entropy. Indeed, since $g$ is constant in this case, the first term of \eqref{eq:firstterm} vanishes, meaning that we can
deduce that $\ep \Theta_{N_n}(N_n) = \log n - H_n + 1/n$ and the entropy is exactly
\begin{equation} \label{eq:entexp} H(N_n) = 1 + H_n - \log n - 1/n, \end{equation}
which value may be of independent interest. Using the fact that $H_n = \log  n + \gamma + 1/(2n) + O(1/n^2)$, we can deduce that
$H(N_n) = 1 + \gamma - 1/(2n) + O(1/n^2)$. In the spirit of \cite{johnson5} and other papers, it may be of interest to ask under what conditions the convergence in Corollary \ref{cor:ent} is at rate $O(1/n)$ in this way. \end{remark}

Theorem \ref{thm:main} shows that convergence in relative entropy occurs for a range of random variables that are `well behaved' in some sense. However, observe that the Gnedenko example $g(u) =  (1-u)^2$ from Example \ref{ex:test} does not satisfy Condition \ref{cond:main2}.\ref{it:tail2}, so Theorem \ref{thm:main} cannot be directly applied in this case. However, it is possible to deduce convergence in relative entropy in this example too, using a relatively simple adaption of the argument to a class of random variables with finite $x_0$ such that the following replacement for Condition \ref{cond:main2}.\ref{it:tail2} holds:

\begin{condition} \label{cond:main2a}
 Assume there exists a constant $\sigma > 1$ such that $\log \left( g(x)/(x_0-x)^\sigma \right)$ is bounded and continuous, and that $\gamma := \lim_{x \rightarrow x_0} g(x)/(x_0-x)^\sigma$  is finite and non-zero.
\end{condition}

\begin{remark} \label{rem:extra}
The only place where we need to adapt the proof of Theorem \ref{thm:main} is in the decomposition of the first term in \eqref{eq:firstterm}, where we can instead use the decomposition:
$$
\log \left( \frac{g(b_n)}{(x_0 - b_n)^\sigma} \right) - \ep \log \left( \frac{g(M_n)}{(x_0 - M_n)^\sigma} \right)  - \sigma \ep \log \left( \frac{x_0 - M_n}{x_0 - b_n} \right) 
$$
As before, the first two terms tend  to $\log \gamma - \log \gamma = 0$ by the portmanteau lemma. We can use a similar Taylor expansion  
$$ \ep \log \left( \frac{x_0 - M_n}{x_0 - b_n} \right) = - \sum_{k=1}^\infty \frac{1}{k} \left( \frac{a_n }{x_0 - b_n} \right)^k \ep N_n^k,$$
and deduce convergence in relative entropy using the fact that $a_n/(x_0 - b_n) = g(b_n)/(x_0-b_n) \simeq \gamma (x_0-b_n)^{1-\sigma} \rightarrow 0$.
\end{remark}

We have shown that there is a natural information-theoretic interpretation of convergence in relative entropy of the standardized maximum to the Gumbel distribution, and provided simple conditions under which this occurs. It would be of interest to provide a similar analysis for the other extreme value distributions -- the Fr\'{e}chet (Type II) and Weibull (Type III) distributions -- which remains an interesting problem for future work, as does the question of the optimal rate of convergence in relative entropy.


\end{document}